\documentclass[11pt]{amsart}
\usepackage{amsmath, amsthm, amssymb}
\usepackage[all]{xy}
\usepackage{mathtools}
\usepackage{enumitem}
\usepackage{mathrsfs}
\usepackage{tikz}
\usepackage{skak}
\usepackage{graphics}
\usepackage{array}

\newcommand{\mcH}{\mathrel{\mathcal{H}}}
\newcommand{\mcR}{\mathrel{\mathcal{R}}}
\newcommand{\mcL}{\mathrel{\mathcal{L}}}
\newcommand{\mcD}{\mathrel{\mathcal{D}}}

\newtheorem{prop}{Proposition}[section]
\newtheorem{thm}[prop]{Theorem}
\newtheorem{cor}[prop]{Corollary}
\newtheorem{lem}[prop]{Lemma}
\theoremstyle{definition}
\newtheorem{defn}[prop]{Definition}

\newtheorem{rem}[prop]{Remark}

\newlist{thmenum}{enumerate}{10}
\setlist[thmenum,1]{label=\textnormal{(\alph*)}}
\setlist[thmenum,2]{label=\textnormal{(\roman*)}}

\begin{document}

\title[Morita Equivalence to Graph Inverse Semigroups]{A note on Morita equivalence to graph inverse semigroups}

\author[M. Du Preez]{Martha Du Preez}
\address{Department of Mathematics\\
The University of Texas at Tyler\\
3900 University Boulevard\\
Tyler, TX 75799}
\email{mdupreez@patriots.uttyler.edu}

\author[R. Grimley]{Robert Grimley}
\address{Department of Mathematics\\
Indiana University\\
Bloomington, IN 47401}
\email{rgrimley@iu.edu}

\author[E. Lira]{Evan Lira}
\address{Department of Mathematics\\
Clarkson University\\
8 Clarkson Ave.\\
Potsdam, NY 13699}
\email{liraem@clarkson.edu}

\author[D. Milan]{David Milan}
\address{Department of Mathematics\\
The University of Texas at Tyler\\
3900 University Boulevard\\
Tyler, TX 75799}
\email{dmilan@uttyler.edu}

\author[S. Ramamurthy]{Shreyas Ramamurthy}
\address{Department of Mathematics\\
University of California, Berkeley\\
970 Evans Hall \#3840\\
Berkeley, CA 94720-3840}
\email{shreyas.ramamurthy@berkeley.edu}

\thanks{The authors were supported by an NSF grant (DMS-2149921).}

\date{\today}

\begin{abstract} We characterize the inverse semigroups that are Morita equivalent to graph inverse semigroups. We also consider a generalization to inverse semigroups associated with left cancellative categories.
\end{abstract}

\maketitle

\section{Introduction}

We characterize the inverse semigroups that are Morita equivalent to graph inverse semigroups. They are the combinatorial inverse semigroups $S$ with $0$ satisfying two additional conditions on the semilattice of idempotents: (1) if $e,f$ are incomparable idempotents with $e,f \leq g$ for some $g \in E(S)$, then $ef = 0$, and (2) for $0 \neq e \leq f$ there are finitely many idempotents between $e$ and $f$ in the natural partial order. A key tool in our proof is a graph $\Gamma_S$ associated with any combinatorial inverse semigroup $S$ satisfying (2). We show that if $S$ is Morita equivalent to some graph inverse semigroup $S(\Gamma)$ then the graph $\Gamma$ must be isomorphic to $\Gamma_S$ (Corollary \ref{cor:uniquegraph}). 

We also consider a generalization of these results to inverse semigroups associated with left cancellative categories. We show that the path category defined in \cite{REU2018} characterizes Morita equivalence of such inverse semigroups.
\section{Preliminaries}
An \emph{inverse semigroup} is a semigroup $S$ such that for each $s$ in $S$ there exists a unique $s^*$ in $S$ such that
\[
	s = ss^*s \quad \text{and}  \quad s^* = s^* s s^*.
\]
The set of idempotents of $S$, denoted $E(S)$, is a commutative subsemigroup of $S$. The natural partial order is defined on $S$ by $s \leq t$ if and only if $s = te$ for some $e \in E(S)$. Green's relations are quite simple to define for inverse semigroups: we have $s \mcL t$ if and only if $s^*s = t^*t$, $s \mcR t$ if and only if $ss^* = tt^*$, and $\mcH \,=\, \mcL \cap \mcR$. Moreover, $s \mcD t$ if and only if there exists $x \in S$ such that $s^*s = x^*x$ and $tt^* = xx^*$. For $e,f \in E(S)$ we have $e \mcD f$ if and only if there exists $x \in S$ with $e = x^*x$ and $f = xx^*$. 

A directed graph $\Gamma = (\Gamma^0, \Gamma^1, r, s)$ consists of sets $\Gamma^0$, $\Gamma^1$ and functions $r,s : \Gamma^1 \to \Gamma^0$ called the \emph{range} and \emph{source} maps, respectively. The elements of $\Gamma^0$ are called \emph{vertices}, and the elements of $\Gamma^1$ are called \emph{edges}. Given an edge $e$, $r(e)$ denotes the range vertex of $e$ and $s(e)$ denotes the source vertex. We denote by $\Gamma^*$ the collection of finite directed paths in $\Gamma$. The range and source maps $r,s$ can be extended to $\Gamma^*$ by defining $r(\alpha) = r(\alpha_n)$ and $s(\alpha) = s(\alpha_1)$ for a path $\alpha = \alpha_n \alpha_{n-1} \cdots \alpha_1$ in $\Gamma^*$. If $\alpha = \alpha_n \alpha_{n-1} \cdots \alpha_1$ and $\beta = \beta_m \beta_{m-1} \cdots \beta_1$ are paths with $s(\alpha) = r(\beta)$, we write $\alpha \beta$ for the path $\alpha_n \cdots \alpha_1 \beta_m \cdots \beta_1$.

The \textit{graph inverse semigroup} of the directed graph $\Gamma$ is the set
\[ S_{\Gamma} = \{(\alpha,\beta) \in \Gamma^* \times \Gamma^* : s(\alpha) = s(\beta) \} \cup \{ 0 \} \]
with products defined by
\[ (\alpha,\beta) (\mu,\nu) = \left\{\begin{array}{ll}
        (\alpha \mu', \nu) & \mbox{if $\mu = \beta \mu'$} \\
        (\alpha,\nu\beta') & \mbox{if $\beta = \mu \beta'$} \\
        0 & \mbox{otherwise}
   \end{array} \right. \]
The inverse is given by $(\alpha,\beta)^{*} = (\beta,\alpha)$. 

Let $T$ be an inverse subsemigroup of an inverse semigroup $S$. Lawson \cite{LawsonEnlargement} defined $S$ to be an \emph{enlargement} of $T$ if $STS = S$ and $TST = T$ and made the case that the concept of enlargement should be part of a larger theory of Morita equivalence for inverse semigroups. Later, Steinberg \cite{SteinbergMorita} developed a general theory of Morita equivalence of inverse semigroups which is motivated by similar definitions in the theory of $C^*$-algebras and depends on the notion of a Morita context. 

\begin{defn} A \emph{Morita context} consists of a 5-tuple $(S,T,X, \langle,\rangle, [,])$ where $S$ and $T$ are inverse semigroups, $X$ is a set equipped with a left action by $S$ and a right action by $T$ that commute, and
\[
	\langle \;,\,\rangle: X \times X \to S, \quad [\;, \,]: X \times X \to T
\]
are surjective functions satisfying the following for $x,y,z$ in $X$, $s$ in $S$, and $t$ in $T$:
\begin{enumerate}

\item $\langle sx, y \rangle = s\langle x, y \rangle$,

\item $\langle y, x \rangle = \langle x, y \rangle^*$,

\item $\langle x, x \rangle x = x$,

\item $[x, yt] = [x,y]t$,

\item $[y,x] = [x,y]^*$,

\item $x[x,x] = x$, and

\item $\langle x, y \rangle z = x [y,z]$.

\end{enumerate}
\end{defn}    
We say that $S$ and $T$ are \emph{strongly Morita equivalent} if there exists a Morita context $(S,T,X, \langle,\rangle, [,])$. There is a useful characterization of this concept in terms of category equivalence. The \emph{idempotent splitting} (also called the Cauchy completion or the Karoubi envelope) of an inverse semigroup $S$ is a category $C(S)$ with objects $E(S)$ and morphisms $\{(e,s,f): e,f \in E(S), s \in eSf\}$. Composition is given by
\[
	(e,s,f)(f,t,g) = (e, st, g).
\]

An important subcategory is $L(S)$ which has the same set of objects but only the morphisms of the form $(f,s,s^*s)$ where $ss^* \leq f$. Note that $s \in eSf$ if and only if $ss^* \leq e$ and $s^*s \leq f$. The isomorphisms in $C(S)$ (and in $L(S)$) are the elements of the form $(s^*s, s, ss^*)$. It follows that two objects $e,f \in E(S)$ are isomorphic in either category if and only if $e \mathcal{D} f$. The following theorem was proved in \cite{FunkLawsonSteinberg} but also relies on \cite[Corollary 5.2]{SteinbergMorita} and results from \cite{FunkTopos}.

\begin{thm}[Funk, Lawson, Steinberg]\label{thm:ME} Let $S$ and $T$ be inverse semigroups. The following are equivalent:
\begin{enumerate}
\item $S$ and $T$ are strongly Morita equivalent.
\item The categories $C(S)$ and $C(T)$ are equivalent.
\item The categories $L(S)$ and $L(T)$ are equivalent.
\end{enumerate}
\end{thm}

We shall say that $S$ and $T$ are \emph{Morita equivalent} if any of the above conditions hold. As alluded to above, if $S$ and $T$ are inverse semigroups where $S$ is an enlargement of $T$, then $S$ is Morita equivalent to $T$. This fact will be important in the proof of our main theorem (Theorem \ref{thm:MEgraph}). 

Morita equivalent inverse semigroups $S$ and $T$ share a number of algebraic properties that will be useful in this paper. For example, given an idempotent $e$ in $S$, there is an idempotent $f$ in $T$ such that $eSe \cong fTf$ \cite{SteinbergMorita}. A semigroup of the form $eSe$ is called a \emph{local submonoid} of $S$. If $P$ is some property of inverse semigroups, we say that \emph{$S$ satisfies $P$ locally} if $eSe$ satisfies $P$ for each idempotent $e$ in $S$. We record here a well-known fact about local submonoids.

\begin{prop}\label{propD} Suppose $S$ is an inverse semigroup and $e,f$ are idempotents in $S$ such that $e \mathcal{D} f$. Then there is an isomorphism $\pi : eSe \to fSf$ such that $s \, \mathcal{D} \, \pi(s)$ for all $s \in eSe$.
\end{prop}

\begin{proof} Choose $x$ in $S$ such that $e = x^*x$ and $f = xx^*$. Then the map $\pi(s) = xsx^*$ satisfies the conclusions of the proposition. 
\end{proof}

\section{Morita Equivalence of Graph Inverse Semigroups}

We want to characterize the inverse semigroups that are Morita equivalent to graph inverse semigroups. Costa and Steinberg proved that two graph inverse semigroups are Morita equivalent if and only if the underlying graphs are isomorphic \cite[Corollary 8.5]{SteinbergCosta}. This result would seem to imply that one can recover the directed graph $\Gamma$ from any inverse semigroup that is Morita equivalent to the graph inverse semigroup $S(\Gamma)$. We start this section by doing just that. Given an inverse semigroup $S$ satisfying some additional conditions on the idempotent semilattice, we define a directed graph $\Gamma_{S}$ whose vertices are the nonzero $\mathcal{D}$-classes of $S$. When $S$ is Morita equivalent to the inverse semigroup of a directed graph $\Gamma$, $\Gamma_{S}$ is isomorphic to $\Gamma$ (Corollary \ref{cor:uniquegraph}). We then use $\Gamma_{S}$ as a tool to characterize those inverse semigroups that are Morita equivalent to graph inverse semigroups (Theorem \ref{thm:MEgraph}). 

Let $S$ be an inverse semigroup with $0$. Jones and Lawson \cite{LawsonGraph} define $S$ to be a \emph{Perrot inverse semigroup} if it satisfies the following properties:

\begin{enumerate}

\item[(P1)] The semilattice of idempotents is unambiguous.

\item[(P2)] For each nonzero idempotent $e$ there are finitely many idempotents above $e$ in the natural partial order.

\item[(P3)] $S$ admits unique maximal idempotents.

\item[(P4)] Each nonzero $\mcD$-class of $S$ contains a maximal idempotent.

\end{enumerate}

Also, $S$ is a \emph{proper} Perrot inverse semigroup if it satisfies the above properties and, in addition, there is a unique maximal idempotent in any nonzero $\mcD$-class. Jones and Lawson obtained the following characterization \cite{LawsonGraph}:

\begin{thm}[Jones, Lawson] The graph inverse semigroups are precisely the combinatorial proper Perrot semigroups.
\end{thm}


The semilattice of an inverse semigroup with $0$ is \emph{unambiguous} if for any idempotents $e,f$ such that $ef \neq 0$, $e$ and $f$ are comparable in the natural partial order. Now, if $S$ is Morita equivalent to a graph inverse semigroup, it must satisfy $(P1)$ locally and $(P2)$ locally. This follows since any local submonoid of $S$ is isomorphic to some local submonoid of the graph inverse semigroup. Notice that we can express the conditions that $S$ satisfies (P1) locally and (P2) locally respectively by:

\begin{enumerate}
\item[(P1L)] for $ e,f,g \in E(S)$ with $e,f \leq g$ and $ef \neq 0$, $e,f$ are comparable.
\item[(P2L)] for $0 \neq e \leq f$, there are finitely many idempotents between $e$ and $f$.
\end{enumerate}

Let $S$ be an inverse semigroup with zero satisfying (P2L). Define a directed graph $\Gamma_{S} = (\Gamma_{S}^0,\Gamma_{S}^1,r,s)$ as follows. First, let $\Gamma_{S}^0$ to be the set of nonzero $\mathcal{D}$-classes of $S$. Denote by $[f]$ the $\mathcal{D}$-class of an idempotent $f$. Next, for each $v \in \Gamma_{S}^0$, choose a nonzero idempotent $e_v \in E(S)$ such that $v = [e_v]$. We write $e \ll f$ and say $e$ \emph{lies directly under} $f$ if and only if $e < f$ and there is no idempotent $g$ with $e < g < f$. Since $e_v S e_v$ satisfies (P2), $e_v^{\downarrow} - \{e_v\} = \{f \in E(S): f < e_v\}$ contains a set 
\[
A_v = \{f \in E(S): 0 \neq f \ll e_v \}
\]
of nonzero maximal idempotents. Moreover, every nonzero idempotent strictly below $e_v$ lies under some element of $A_v$. For each $f \in A_v$ we include an edge $x_{v,f}$ with source $[f]$ and range $v$. That is, 
\[
  \Gamma_{S}^1 = \bigcup_{v \in {\Gamma_{S}^0}} \{ x_{v,f} : f \in A_v \}
\]
and for each $x_{v,f} \in \Gamma_{S}^1$, we define $s(x_{v,f}) = [f]$ and $r(x_{v,f}) = v$. In this construction we made a choice of representative $e_v$ for each $\mathcal{D}$-class $v$. By Proposition \ref{propD}, the number of idempotents $f$ with $f \ll e_v$ and their corresponding $\mathcal{D}$-classes are the same regardless of the choice of $e_v$. Thus the directed graph $\Gamma_S$ does not depend on the choice of representatives.

Now, suppose that $S$ is a combinatorial inverse semigroup with $0$ satisfying (P2L). Based on these assumptions we define a set $\{s_{\alpha} : \alpha \in \Gamma_{S}^*\}$ in $S$ as follows:
\begin{enumerate}

\item for $v \in \Gamma_S^{0}$, $s_v:=e_v$.

\item for each edge $x = x_{v,f}$ in $\Gamma_S^{1}$, $s_x$ is defined to be the unique element of $S$ such that $s_x^* s_x = e_{[f]}$ and $s_x s_x^* = f$; and

\item for each path $\alpha = x_1 x_2 \cdots x_n$, define $s_\alpha := s_{x_1} s_{x_2} \cdots s_{x_n}$. 

\end{enumerate} 

Also, given a path $\alpha$, define $e_{\alpha} := s_{\alpha} s_{\alpha}^*$. First, we introduce some important lemmas:

\begin{lem}\label{lem:graphidemp}
Let $S$ be a combinatorial inverse semigroup that satisfies (P2L) and let $\Gamma_S$ be the associated directed graph. Define $\{s_{\alpha} : \alpha \in \Gamma_{S}^*\}$ as above. For $\alpha \in \Gamma_S$, we have $s_\alpha^* s_\alpha=e_{s(\alpha)}$. Also, for each nonzero idempotent $e \leq e_v$, there is a path $\alpha$ in $\Gamma_S^*$ such that $e = e_{\alpha}$ and $r(\alpha) = v$.
\end{lem}

\begin{proof}
Let $\alpha = x_1 x_2 \cdots x_n$ be a path where $x_i = x_{v_i, f_i}$ and $v_{i+1} = [f_i]$ for $i \leq n-1$. Notice that $s_{x_i}^* s_{x_i} = e_{[f_i]} = e_{v_{i+1}}$. Also, since $ s_{x_{i}}s_{x_{i}}^* \leq e_{v_{i}}$ we have $s_{x_{i}}^* e_{v_{i}}s_{x_{i}} = s_{x_{i}}^* s_{x_{i}}$ for all $i$. So
\begin{align*}
s_{\alpha}^* s_{\alpha} &= s_{x_n}^* \cdots s_{x_2}^* (s_{x_1}^* s_{x_1}) s_{x_2} \cdots s_{x_n}\\
&= s_{x_n}^* \cdots s_{x_2}^* (e_{v_2}) s_{x_2} \cdots s_{x_n} \\
&= s_{x_n}^* \cdots (s_{x_2}^* s_{x_2}) \cdots s_{x_n} \\
&\, \cdots \\
&= s_{x_n}^* s_{x_n} = e_{s(\alpha)}
\end{align*}

For the second claim, suppose that $0 \neq e \leq e_v$ for some $v \in \Gamma^0_{S}$. By (P2L) there are idempotents $e_i$ for $1 \leq i \leq n$ such that $e = e_1 \ll e_2 \ll \dots e_{n-1} \ll e_n = e_v$. We will induct on $n$. In the case $n=1$, $e = e_v$ and we are done. For $n > 1$ we assume that $e_2 = e_{\alpha'}$ where $\alpha' \in \Gamma_S^*$ with $r(\alpha') = v$. Let $u = s(\alpha')$. As in Proposition \ref{propD}, the map $s \mapsto s_{\alpha'}^*s{s_{\alpha'}}$ defines an isomorphism from $e_2 S e_2$ to $e_u S e_u$. In particular, $f := s_{\alpha'}^* e {s_{\alpha'}} \ll e_u$. Thus $x_{u,f} \in \Gamma^1_S$ with $r(x_{u,f}) = u$. We have $\alpha =  \alpha' x_{u,f} \in {\Gamma_{S}}^*$ with $r(\alpha) = v$. Finally,
\[
	e_{\alpha} = s_{\alpha} s_{\alpha}^* =  s_{\alpha'} s_{x_{u,f}} s_{x_{u,f}}^* s_{\alpha'}^* = s_{\alpha'} f s_{\alpha'}^* =  s_{\alpha'} s_{\alpha'}^* e {s_{\alpha'}} s_{\alpha'}^* = e
\]
since $e \leq e_2 = s_{\alpha'} s_{\alpha'}^*$.

\end{proof}

As a consequence of the lemma we note that for each path $\alpha \in \Gamma_S^{*}$ we have $s_\alpha \neq 0$ and $e_\alpha \neq 0$, since otherwise $e_{s(\alpha)} = s_\alpha^* s_\alpha = 0$. 

For what follows we will need the additional assumption that there is an orthogonal set of idempotent representatives of the nonzero $\mathcal{D}$-classes of $S$. That is, $e_u e_v = 0$ for $u \neq v$. Note that it is not possible to choose such a set for general $S$, but we will overcome this difficulty later.

\begin{lem}\label{lem:multiplication}Let $S$ be a combinatorial inverse semigroup that satisfies (P1L) and (P2L), and suppose that $\{e_v: v \in \Gamma_S^{0}\}$ is an orthogonal set of idempotent representatives of the nonzero $\mathcal{D}$-classes of $S$. Let $\Gamma_S$ be the associated directed graph, and let $\{s_{\alpha} : \alpha \in \Gamma_{S}^*\}$ be as defined above. Then for $\alpha,\beta$ be in $\Gamma^{*}$,
 \[ s_{\alpha}^*s_\beta = \left\{\begin{array}{ll}
        s_{\beta'} & \mbox{if $\beta = \alpha \beta'$} \\
        s_{\alpha'}^*  & \mbox{if $\alpha = \beta \alpha'$}\\
        0             & \mbox{otherwise}\\ \end{array} \right. \]
\end{lem}
\begin{proof}

Suppose $\alpha,\beta$ in $\Gamma^{*}$. If $\beta=\alpha\beta'$, then $s_\beta=s_\alpha s_{\beta'}=s_\alpha(s_\alpha^* s_\alpha)(s_{\beta'}s_{\beta'}^*)s_{\beta'}$. As remarked above, $s_{\beta} \neq 0$, so $e_{s(\alpha)}(s_{\beta'}s_{\beta'}^*) = (s_\alpha^* s_\alpha)(s_{\beta'}s_{\beta'}^*)\neq 0$. We have $s_{\beta'}s_{\beta'}^* \leq e_{r(\beta')}$. Then $e_{r(\beta')}e_{s(\alpha)} \neq 0$ and by orthogonality $e_{r(\beta')} = e_{s(\alpha)}$. Thus $s_{\beta'}s_{\beta'}^* \leq s_\alpha^* s_\alpha$. So $s_\alpha^* s_\beta = s_\alpha^* s_\alpha s_{\beta'}=(s_\alpha^* s_\alpha)(s_{\beta'}s_{\beta'}^*)s_{\beta'}=(s_{\beta'}s_{\beta'}^*)s_{\beta'}=s_{\beta'}$

The case where $\alpha = \beta \alpha'$ is similar.

For the last case, first suppose that $x = x_{v,f} ,x' = x_{u,g} \in \Gamma_{S}^{1}$ are edges such that $s_x^*s_{x'}\not=0$. We will show that $x = x'$. Notice that $(s_x s_x^*)(s_x' s_x'^*) \neq 0$, so $f \ll e_v$, $g \ll e_u$, and $fg \neq 0$. Since we chose an orthogonal set of representatives, $u = v$. Since $S$ satisfies (P1L), $f$ and $g$ are comparable. But both idempotents lie directly under $e_u$, so $f = g$. Thus $x = x'$.

Now suppose $\alpha$ is not a subpath of $\beta$ and $\beta$ is not a subpath of $\alpha$. The cases where one or both of $\alpha$ and $\beta$ are trivial paths (vertices) is left to the reader. We assume $\alpha$ and $\beta$ are nontrivial paths.  Then for some paths $w,q,q'$ and distinct edges $x\not=x'$, we have $\alpha=wxq$ and $\beta=wx'q'$. Then,
\[s_\alpha^*s_\beta = s_q^* s_x^* (s_w^* s_w) s_{x'} s_{q'} \le s_q^* s_x^* s_{x'} s_{q'}= s_q^*(s_x^* s_{x'})s_{q'}=0,\]
as desired.
\end{proof}

\begin{lem}\label{lem:uniquepaths} Let $S$ be a combinatorial inverse semigroup that satisfies (P1L) and (P2L), and suppose that $\{e_v: v \in \Gamma_S^{0}\}$ is an orthogonal set of idempotent representatives of the nonzero $\mathcal{D}$-classes of $S$. Let $\Gamma_S$ be the associated directed graph, and let $\{s_{\alpha} : \alpha \in \Gamma_{S}^*\}$ be as defined above. If $\alpha,\beta$ are paths in $\Gamma_{S}^*$, then $\alpha=\beta$ if and only if $e_\alpha=e_\beta$.
\end{lem}

\begin{proof}
If $\alpha=\beta$, then $e_\alpha=e_\beta$ by definition. Suppose for paths $\alpha$ and $\beta$ that $e_\alpha=e_\beta$. Then
 \[e_\alpha=e_\alpha e_\beta=s_\alpha s_\alpha^* s_\beta s_\beta^*\not=0,\]
so $s_\alpha^* s_\beta\not=0$ and $\alpha$ is a subpath of $\beta$ or $\beta$ is a subpath of $\alpha$. For the sake of contradiction, assume $\alpha$ is a proper subpath of $\beta$. That is, $\beta=\alpha x_1x_2\dots x_k$ for some edges $x_1,x_2,\dots,x_k$. From, $e_\alpha=e_\beta$, we have $s_\alpha s_\alpha^* = s_\beta s_\beta^*$ and hence
\begin{align*}
e_{s(\alpha)} = s_\alpha^* s_\alpha &=s_\alpha^*(s_\beta s_\beta^*)s_\alpha=(s_\alpha^* s_\beta)(s_\beta^* s_\alpha)=s_{x_1x_2\dots x_k}s_{x_1x_2\dots x_k}^*\\&=s_{x_1}s_{x_2}\dots s_{x_k}s_{x_k}^*s_{x_{k-1}}^*\dots s_{x_1}^* \le s_{x_1}s_{x_1}^* \ll e_{r(x_1)}
\end{align*}
This is a contradiction, since by our hypothesis on the $\mathcal{D}$-class representatives we know that either $e_{s(\alpha)} = e_{r(x_1)}$ or $e_{s(\alpha)} e_{r(x_1)} = 0$. The other case is similar, so we have shown that $\alpha=\beta$.
\end{proof}

\begin{thm}\label{thm:graphembed} Let $S$ be a combinatorial inverse semigroup that satisfies (P1L) and (P2L), and suppose that $\{e_v: v \in \Gamma_S^{0}\}$ is an orthogonal set of idempotent representatives of the nonzero $\mathcal{D}$-classes of $S$. Then the map $ (\alpha, \beta) \mapsto s_{\alpha}s_{\beta}^*$ and $0 \mapsto 0$ defines an embedding of the graph inverse semigroup of $\Gamma_S$ into $S$. 
\end{thm}

\begin{proof} First we show that the map $\pi : S_{\Gamma_{S}} \to S$ defined by $\pi( (\alpha, \beta) ) = s_{\alpha}s_{\beta}^*$ and $\pi(0) = 0$ is a homomorphism. Let $(\alpha, \beta), (\mu, \nu) \in S_{\Gamma_{S}}$. It follows from Lemma \ref{lem:multiplication} that 
\[ s_{\alpha} s_{\beta}^* s_{\mu} s_{\nu}^* = \left\{\begin{array}{ll}
       s_{\alpha \mu'} s_{\nu}^* & \mbox{if $\mu  = \beta \mu'$} \\
         s_{\alpha} s_{\nu \beta'}^* & \mbox{if $\beta = \mu \beta'$} \\
         0 & \mbox{otherwise} \end{array} \right. \]
Comparing this with the multiplication operation in a graph inverse semigroup (see section 2), we quickly see that $\pi$ is a homomorphism. Let $t_1 = s_{\alpha} s_{\beta}^*$ and $t_2 = s_{\mu} s_{\nu}^*$ and suppose that $t_1 = t_2$. Notice 
\[
t_1^* t_1 = s_{\beta} s_{\alpha}^* s_{\alpha} s_{\beta}^* = s_{\beta} e_{s(\alpha)} s_{\beta}^* = s_{\beta} e_{s(\beta)} s_{\beta}^* = s_{\beta} s_{\beta}^* = e_{\beta}
\] 
since $s_{\beta}^* s_{\beta} = e_{s(\beta)}$. Similarly $t_1 t_1^* = e_{\alpha}$, $t_2^* t_2 = e_{\nu}$, and $t_2 t_2^* = e_{\mu}$. We have $e_{\alpha} = e_{\mu}$ and $e_{\beta} =  e_{\nu}$. Thus $\alpha = \mu$ and $\beta = \nu$ by Lemma \ref{lem:uniquepaths}. Therefore $\pi$ is injective.

\end{proof}

Of course, not every inverse semigroup will admit an orthogonal set of idempotent representatives of its nonzero $\mathcal{D}$-classes. To overcome this obstacle, we work with an enlarged version of $S$. Given an inverse semigroup $S$ with $0$, define
\[
	S' = \{(e,s,f) : ss^* \leq e, s^*s \leq f, \text{ and $s \neq 0$}\} \cup \{0\}
\]
with multiplication given by 
\[ (e,s,f)(g,t,h) = \left\{\begin{array}{ll}
       (e,st,h) & \mbox{if $f = g$ and $st \neq 0$} \\
         0 & \mbox{otherwise} \end{array} \right. \]

Then $S'$ is an inverse semigroup with $0$ that is Morita equivalent to $S$ (See section 6 of \cite{REU2018}). The nonzero idempotents of $S'$ are exactly $(e,f,e)$ such that $0 \neq f \leq e$. Moreover, if $\{e_v: v \in \Gamma^0\}$ is a collection of idempotent representatives of the nonzero $\mathcal{D}$-classes of $S$, then $\{(e_v, e_v, e_v): v \in \Gamma^0\}$ is an orthogonal set of representatives of the nonzero $\mathcal{D}$-classes of $S'$. Moreover, each representative is maximal in the natural partial order. Notice that $S'$ has an additional useful property: given two idempotents $x,y$ in $S'$ with nonzero product, there is an idempotent $z$ in $S'$ such that $x,y \leq z$.
 
We are now prepared to prove our main theorem.

\begin{thm}\label{thm:MEgraph} An inverse semigroup $S$ is Morita equivalent to a graph inverse semigroup if and only if $S$ is combinatorial, has a $0$, and satisfies (P1L) and (P2L).
\end{thm}

\begin{proof}
As discussed at the beginning of this section, any inverse semigroup that is Morita equivalent to a graph inverse semigroup satisfies (P1L) and (P2L). Such semigroups must also be combinatorial and contain a zero by \cite[Corollary 5.2]{SteinbergMorita}. 

Suppose that $S$ is combinatorial, has a $0$, and satisfies (P1L) and (P2L). Then $S'$ is Morita equivalent to $S$ and hence satisfies the same hypotheses. Moreover, $S'$ contains an orthogonal set of idempotent representatives $\{e_v : v \in \Gamma^0_{S'}\}$ of its nonzero $\mathcal{D}$-classes such that each $e_v$ is maximal in the natural partial order. It suffices to show that $S'$ is Morita equivalent to a graph inverse semigroup. By Theorem \ref{thm:graphembed}, $T = \{s_{\alpha}s_{\beta}^* : (\alpha,\beta) \in S_{\Gamma_{S'}}\} \cup \{0\}$ is an inverse subsemigroup of $S'$ that is isomorphic to the graph inverse semigroup of $\Gamma_{S'}$. 

We show that $T$ is Morita equivalent to $S'$ by showing that $S'$ is an enlargement of $T$. That is, we prove that $S' = S'TS'$ and $T = TS'T$. Note that the containments $S'TS' \subseteq S'$ and $T \subseteq TS'T$ are immediate. Next, suppose that $s \in S'$. If $s=0$ then we know $s \in S'TS'$. Otherwise there exists $v \in \Gamma^0_{S'}$ and $x \in S'$ such that $ss^* = xx^*$ and $e_v = x^*x$. But then $s = ss^* s s^* s = x x^* x x^* s = x (e_v) x^* s \in S'TS'$. 

Next, let $y = t_1 s t_2$ with $t_1, t_2 \in T$ and $s \in S'$. Write $t_1 = s_{\alpha} s_{\beta}^*$ and $t_2 = s_{\mu} s_{\nu}^*$. If $y = 0$, then $y \in T$. Otherwise we have $s_{\beta} s_{\beta}^* s s^* \neq 0$. Thus $e_{r(\beta)} s s^* \neq 0$. As remarked before the theorem, there is an idempotent $z$ such that $e_{r(\beta)}, s s^* \leq z$. As $e_{r(\beta)}$ is a maximal idempotent in $S'$, $e_{r(\beta)} = z$ and we have that $ss^* \leq e_{r(\beta)}$. By Lemma \ref{lem:graphidemp}, there is a path $\epsilon \in \Gamma^*_{S'}$ with $r(\epsilon) = r(\beta)$ such that $ss^* = s_{\epsilon}s_{\epsilon}^*$. Similarly there exists a path $\delta$ with $r(\delta) = r(\mu)$ such that $s^* s = s_{\delta}s_{\delta}^*$. Thus $s \mathcal{H} s_{\epsilon} s_{\delta}^*$. As $S'$ is combinatorial, we have $s = s_{\epsilon}s_{\delta}^* \in T$ and hence $y \in T$.

Therefore $S$ is Morita equivalent to a graph inverse semigroup.

\end{proof}

\begin{rem} Though we proved in the last theorem that $S$ is Morita equivalent to the inverse semigroup of the graph $\Gamma_{S'}$, we note that this directed graph is the same as $\Gamma_{S}$. To see this, first notice that the map $s \mapsto (e_v, s, e_v)$ defines an isomorphism from $e_v S e_v$ to $(e_v, e_v, e_v) S' (e_v, e_v, e_v)$. Moreover for $s,t$ in $e_v S e_v$, $s\mathcal{D}t$ if and only if $(e_v, s, e_v) \mathcal{D} (e_v, t, e_v)$. Since the directed graph of an inverse semigroup is defined based on the nonzero $\mathcal{D}$-class representatives and the $\mathcal{D}$-classes of the idempotents directly below each representative, we have that $\Gamma_S$ is isomorphic to $\Gamma_{S'}$.
\end{rem}

\begin{cor}\label{cor:uniquegraph} Let $S$ be an inverse semigroup that is Morita equivalent to the graph inverse semigroup $S_{\Gamma}$. Then $\Gamma_S$ is isomorphic to $\Gamma$.
\end{cor}

\begin{proof} By Theorem \ref{thm:MEgraph} and the remark that follows the theorem, $S_{\Gamma_S}$ is Morita equivalent to $S_{\Gamma}$. Thus by \cite[Corollary 8.5]{SteinbergCosta} we have $\Gamma_S$ is isomorphic to $\Gamma$. 
\end{proof}

Finally, we note a useful consequence of the results in this section that allow us to quickly check Morita equivalence of two inverse semigroups. 

\begin{cor} Let $S$ and $T$ be combinatorial inverse semigroups with $0$ that satisfy (P1L) and (P2L). Then $S$ is Morita equivalent to $T$ if and only if $\Gamma_S$ is isomorphic to $\Gamma_T$.
\end{cor}

\section{Path Categories}

In this section we consider a generalization of graph inverse semigroups to those semigroups associated with left cancellative categories. The case was made in \cite{REU2018} that the path category plays a role similar to that of the paths in the directed graph for inverse semigroups that admit unique maximal idempotents. To what extent do the results in the previous section extend to these inverse semigroups? In fact, every inverse semigroup with $0$ is Morita equivalent to an inverse semigroup satisfying (P3) and (P4) \cite[Corollary 6.2]{REU2018}. Also, any semigroup satisfying (P3) and (P4) is isomorphic to the canonical inverse semigroup associated with its path category. So Theorem \ref{thm:MEgraph} has the rather decisive (but not new) generalization that every inverse semigroup is Morita equivalent to an inverse semigroup associated with a left cancellative category. It is interesting to determine whether the path category alone can be used to characterize Morita equivalence for such semigroups. 

If $S$ is an inverse semigroup satisfying (P3), then for each nonzero idempotent $e$ there is a unique maximal idempotent $e^{\circ}$ such that $e \leq e^{\circ}$. We recall the definition of the path category $P(S)$ from \cite{REU2018}. The objects are the maximal idempotents of $S$. The arrows in $P(S)$ are the pairs $(e, s)$ such that $s^*s$ is maximal and $({ss^*})^{\circ} = e$. Composition is given by
\[
	(e,s)(f,t) = (e,st)
\]
provided $s^*s = f$.

\begin{prop} Suppose $S$ is an inverse semigroup with $0$ satisfying (P3) and (P4). Then $P(S)$ is equivalent to $L(S)$.
\end{prop}
\begin{proof}
It is straightforward to verify that $P(S)$ embeds in $L(S)$ as the subcategory of triples $(e,s,s^*s)$ where $e$ and $s^*s$ are maximal idempotents. As remarked in the preliminaries section, two objects $e,f$ in $L(S)$ are isomorphic if and only if $e \mathcal{D} f$. Therefore, by (P4), the inclusion of $P(S)$ into $L(S)$ is an equivalence of categories.
\end{proof}

Thus we quickly derive the following corollary as a consequence of the result of Funk, Lawson, and Steinberg (see Theorem \ref{thm:ME}).

\begin{cor} Suppose $S$ and $T$ are inverse semigroups with $0$ satisfying (P3) and (P4). Then $S$ is Morita equivalent to $T$ if and only if $P(S)$ is equivalent to $P(T)$.
\end{cor}

%
%
%

\bibliographystyle{amsplain}
\bibliography{SemigroupBib.bib}
\end{document}